\documentclass[12pt]{amsart}
\newtheorem{theorem}{Theorem}[section]

\newtheorem{corollary}[theorem]{Corollary}
\newtheorem{proposition}[theorem]{Proposition}
\theoremstyle{definition}
\newtheorem{definition}[theorem]{Definition}
\newtheorem{example}[theorem]{Example}

\usepackage{enumerate}
\newtheorem{remark}[theorem]{Remark}
\usepackage[margin=0.75in]{geometry}
\usepackage{calligra}
\usepackage{caption}
\usepackage{tikz}
\usetikzlibrary{graphs,graphs.standard}
\usepackage{xcolor, soul}
\sethlcolor{yellow}
\usepackage{mathtools}

\DeclareMathOperator{\sinc}{sinc}

\DeclareMathOperator\Span{span}

\numberwithin{equation}{section}

\def\R{{\mathbb{C}^{N}}}
\def\d{{\delta}}

\def\I{{\{1,...,N\}}}
\def\c{{\chi}}
\def\V{{\{T_{i}g : 1\leq i \leq m\}}}

\newcommand\numberthis{\addtocounter{equation}{1}\tag{\theequation}}

\newcommand{\be}{\begin{equation}}
	\newcommand{\ee}{\end{equation}}

\begin{document}
	\title{Frames and bases of translates of signals on undirected graphs}
	
	
	\author{Rabeetha Velsamy}
	\address{Department of Mathematics, Indian Institute of Technology, Madras, India}
	\email{rabeethavelsamy@gmail.com}
	
	\author{Radha Ramakrishnan}
	\address{Department of Mathematics, Indian Institute of Technology, Madras, India}
	\email{radharam@iitm.ac.in}

	
	
	
	\begin{abstract}
		We study a shift invariant space on an undirected graphs $G$ having $N$ vertices. We obtain a characterization theorem for a system of generalized translates $\{T_{i}g : 1\leq i\leq N\}$, for $g\in \mathbb{C}^{N}$, to form an orthonormal basis. Moreover, we find a necessary and sufficient condition for the system $\{T_{i}g : 1\leq i\leq m\}$, $m\leq N$, to form a linearly independent set and an orthonormal set. Further, we obtain a charactreization result for a system of generalized translates which is generated by multiple generators $g_{1},...,g_{M}$ to form a frame for $\mathbb{C}^{N}$. In particular, we deduce similar results for the system $\{T_{i}M_{s}g : 1\leq i,s\leq N\}$ with modulation $M_{s}$ and the spectral graph wavelet system. We also provide an illustration for the spectral graph wavelet system.
	\end{abstract}
	\maketitle
	\section{Introduction}
	The explosive growth of data (big data) which often exhibits unstructured patterns with different data formats, arising out of diverse data sources, in the current technology requires strong theoretical aspects apart from classical signal processing. The current day applications involving wide variety of network based data such as economic and social network, sensor network, traffic flows and so on naturally lead to graph structures. Hence in order to represent and analyse the relationship between the nodes of large network, it becomes essential to study signal processing on graphs. We refer to \cite{DSP_moura} to know the fundamental concepts of graph signal processing (GSP). The mathematical foundation of GSP is based on applied and computational harmonic analysis and spectral graph theory. We refer to \cite{chen,chengcheg,intro} in this direction.
	
	In GSP, graph Fourier transform is a fundamental tool which helps in decomposing a graph signal into various frequency components and representing them by different modes of variation. The classical Fourier transform can be viewed as an expansion of a function in terms of the eigenfunctions of the Laplacian on the Real line. In analogous way, the graph Fourier transform of a graph signal on the vertices of a graph is the expansion of the graph signal in terms of the eigenvectors of the graph Laplacian or more generally eigenvectors of a matrix (such as adjacency matrix, graph Laplacian and so on) associated with a graph whose eigenvectors form an orthonormal basis for $\mathbb{R}^{N}$ (or $\mathbb{C}^{N}$). For a study of graph Fourier transform on undirected graphs, we refer to \cite{SGt,Fastgraph,l1norm} and for directed graphs \cite{qsun}. In classical signal processing, the convolution theorem plays a central role, namely convolution in the time domain of a signal is equivalent to multiplication in the frequency (Fourier) domain. Using this notion, one can define the notion of convolution for graph signals by making use of the graph Fourier transform, which in turn leads to a generalized translation operator, (see \cite{MR3440174}).
	
	 In classical Fourier analysis, using the system of integer translations, one can define an important space, called shift invariant space in the class of square integrable functions $L^{2}(\mathbb{R})$ on $\mathbb{R}$. The principal shift invariant space is defined to be the closure of span of integer translates of $\phi\in L^{2}(\mathbb{R})$. This space dates back to the work of classical Shannon sampling theorem which makes use of bandlimited functions in $L^{2}(\mathbb{R})$ with band in $[-\frac{1}{2},\frac{1}{2}]$. This space coincides with the shift invariant space with $\phi=\sinc$, where $\sinc(x)=\frac{sin(\pi x)}{\pi x}$. These shift invariant spaces with various functions $\phi$ play an important role in the construction of wavelets. We refer to the classical works of Meyer \cite{meyer} and Mallat \cite{mallat} in this direction. One of the interesting and important question in Fourier analysis is to find necessary and sufficient conditions for the system of translates to form a frame or an orthonormal basis or a Riesz basis. Frames in a Hilbert space are generalization of orthonormal basis and Riesz basis. Frames were introduced by Duffin and Schaeffer in \cite{duffin} in
	order to study non-harmonic Fourier series. Frames are very useful in sampling and quantization theory, communication theory and so on. In the case of GSP, two important classes of frames, namely Gabor frames and wavelet frames have been investigated. We refer to the works \cite{ghandehari} and \cite{hammond} and the references therein.
	
	In this paper, we consider a diagonalizable matrix $A$ which is associated with a graph G having $N$ vertices. We define a shift invariant space on graphs and prove that $\text{span}\{T_{i}g : 1\leq i\leq N\}$ is a shift invariant space, where $g\in \mathbb{C}^{N}$. We also consider a system of generalized translates which is generated by a single genertaor $g\in \mathbb{C}^{N}$ and obtain its characterization to form an orthonormal basis. Moreover, we obtain a necessary and sufficient condition for the system $\{T_{i}g : 1\leq i\leq m\}$, $m\leq N$, to be a linearly independent set and an orthonormal set. Finally, we study a shift invariant space which is generated by multiple generators $g_{1},...,g_{M}$ and obtain a characterization theorem for a system of generalized translates with multiple generators to form a frame for $\mathbb{C}^{N}$. In particular, we deduce similar results for the system $\{T_{i}M_{s}g : 1\leq i,s\leq N\}$ with modulation $M_{s}$ and the spectral graph wavelet system. We also provide an illustration for the spectral graph wavelet system. 
	
	\section{Notation and Background}
	Let $0\neq \mathcal{H}$ be a separable Hilbert space.
	\begin{definition}
		A family of elements $\{f_{k} : k\in \mathbb{N}\}$ in $\mathcal{H}$ is called a frame for $\mathcal{H}$ if there exist two constants $A,B>0$ such that
		\begin{align*}
			A\|f\|^{2}\leq \sum_{k\in \mathbb{N}}|\langle f,f_{k}\rangle|^{2} \leq B\|f\|^{2},\hspace{10mm}\forall\ f\in \mathcal{H}.\numberthis \label{frm}
		\end{align*}
	\end{definition}
	If \eqref{frm} holds with $A=B=1$, then $\{f_{k} : k\in \mathbb{N}\}$ is called a Parseval frame. The frame operator $S:\mathcal{H}\rightarrow \mathcal{H}$ associated with a frame $\{f_{k} : k\in \mathbb{N}\}$ is defined by
	\begin{align*}
		Sf:=\sum\limits_{k\in \mathbb{N}}\langle f,f_{k}\rangle f_{k},\hspace{10mm}\forall\ f\in \mathcal{H}.
	\end{align*}
	It can be shown that $S$ is an invertible, self-adjoint and positive operator on $\mathcal{H}$. In addition, $\{S^{-1}f_{k} : k\in \mathbb{N}\}$ is also a frame with frame operator $S^{-1}$ and frame bounds $B^{-1}, A^{-1}$.
	\begin{definition}
		Let $\{f_{k} : k\in \mathbb{N}\}$ be a frame for $\mathcal{H}$ and $S$ be the corresponding frame operator. Then the collection $\{S^{-1}f_{k} : k\in \mathbb{N}\}$ is called the canonical dual frame of $\{f_{k} : k\in \mathbb{N}\}$.
	\end{definition}
	\begin{definition}
		Let $\{f_{k} : k\in \mathbb{N}\}$ and $\{g_{k} : k\in \mathbb{N}\}$ be two frames for $\mathcal{H}$. Then $\{g_{k} : k\in \mathbb{N}\}$ is said to be a dual frame of $\{f_{k} : k\in \mathbb{N}\}$ if
		\begin{align*}
			f=\sum_{k\in \mathbb{N}}\langle f, g_k\rangle f_k,\ \ \ \ \ \ \ \forall\ f\in \mathcal{H}.
		\end{align*}
	\end{definition}
	\begin{theorem}\label{dualthm}
		Let $\{f_{k} : k\in \mathbb{N}\}$ and $\{g_{k} : k\in \mathbb{N}\}$ be two subsets of $\mathcal{H}$. Then the following are equivalent:
		\begin{enumerate}[(i)]
			\item $f=\sum\limits_{k\in\mathbb{N}}\langle f,g_k\rangle f_k,\hspace{3 mm} \forall\ f\in \mathcal{H}$.
			\vspace{1 mm}
			\item $f=\sum\limits_{k\in\mathbb{N}}\langle f,f_k\rangle g_k,\hspace{3 mm} \forall\ f\in \mathcal{H}$.
			\vspace{1 mm}
			\item $\langle f,g\rangle =\sum\limits_{k\in\mathbb{N}}\langle f,f_k\rangle\langle g_k,g\rangle,\hspace{3 mm}\forall\ f,g\in \mathcal{H}.$
		\end{enumerate}
		When one of the above equivalent conditions is satisfied, $\{f_k:k\in \mathbb{N}\}$ and $\{g_k:k\in \mathbb{N}\}$ are dual frames for $\mathcal{H}$. If $B$ denotes an upper frame bound for $\{f_k:~k\in \mathbb{N}\}$, then $B^{-1}$ is a lower frame bound for $\{g_k:~k\in \mathbb{N}\}$.
	\end{theorem}
	\begin{definition}
		Let $\{f_{k} : k\in \mathbb{N}\}$ and $\{g_{k} : k\in \mathbb{N}\}$ be two sequences in $\mathcal{H}$. They are said to be biorthogonal if they satisfy the relation $$\langle f_{k}, g_{l} \rangle =\delta_{kl}.$$
	\end{definition}
	For further study of frames we refer to \cite{CN}.\\
	
	Now, we shall provide the necessary background on spectral graph theory and graph Fourier transform.\\
	
	Let $G$ be an undirected, connected, weighted graph with a vertex set $\mathcal{V}$ such that $|\mathcal{V}|=N$. Then, its associated weighted adjacency matrix is denoted by $W$. Let $A$ be a diagonalizable matrix associated with the graph G. The Laplacian matrix and the adjacency matrix are some examples of such matrix corresponding to the graph G. 
	Let $\lambda_{1},...,\lambda_{N}$ be eigenvalues of $A$ with the corresponding eigenvectors $\c_{1}, \ldots, \c_{N}$.\\
	
	A signal $f:\mathcal{V}\rightarrow \mathbb{C}$ is a function defined on $\mathcal{V}$, which can be identified as a vector in $\R$ as $(f(1),...,f(N))$, where $f(n)$ denotes the signal value of $f$ at the $n^{th}$ vertex in $\mathcal{V}$. The graph Fourier transform $\widehat{f}$ of a function $f\in \R$ is defined by $$\widehat{f}(\lambda_{l})=\sum\limits_{n=1}^{N}f(n)\overline{\c_{l}(n)},\quad l=1,...,N.$$ Notice that the graph Fourier transform of $f$ is defined on the spectrum of $A$, denoted by $\text{spec}(A)$. The inverse graph Fourier transform is given by $$f(n)=\sum\limits_{l=1}^{N}\widehat{f}(\lambda_{l})\c_{l}(n),\quad n=1,...,N.$$ Since $\{\c_{l}\}_{l=1}^{N}$ is the set of orthonormal eigenvectors, we have the following Parseval identity $$\langle f, g \rangle=\langle \widehat{f}, \widehat{g} \rangle\quad f,g\in \R.$$  
	\begin{definition}
		Let $f,g\in \mathbb{C}^{N}.$ Then, the generalized convolution $f\ast g$ is defined by $$f\ast g(n)=\sum\limits_{l=1}^{N}\widehat{f}(\lambda_{l})\widehat{g}(\lambda_{l})\c_{l}(n),\quad n=1,...,N.$$
	\end{definition}
	\begin{proposition}\label{convpro}
		Let $f,g,h\in \R$. Then, the generalized convolution satifies the following properties:
		\begin{enumerate}
			\item $\widehat{f\ast g}=\widehat{f}~\widehat{g}.$
			\item $\alpha(f\ast g)=(\alpha f)\ast g=f\ast (\alpha g),\quad \text{for all}~\alpha\in \mathbb{C}.$
			\item $f\ast g=g\ast f.$
			\item $f\ast (g+h)=f\ast g+ f\ast h.$
		\end{enumerate}
	\end{proposition} 
	The generalized convolution can also be written in the form of matrix as follows: \begin{equation}\label{conmatri}
		f\ast g=\widehat{g}(A)f=\chi \begin{pmatrix}
			\widehat{g}(\lambda_{1}) & \dots & 0\\
			\vdots & \ddots & \vdots\\
			0 & \dots & \widehat{g}(\lambda_{N}) 
		\end{pmatrix}\chi^{\ast}f,
	\end{equation} 
	where $\chi$ is the $N\times N$ matrix whose $j^{th}$ column is $\chi_{j}$, $j=1,...,N$. 
	\begin{definition}
		Let $f\in \R.$ Then, for $i=1,...,N$, the generalized translation $T_{i}f$ of $f$ is defined by $$T_{i}f(n)=\sqrt{N}(f\ast \delta_{i})(n),\quad n=1,...,N,$$
		where $\delta_{i}$ is the Kronecker delta function on $G$. It is defined by \begin{equation*}
			\delta_{i}(n)=\begin{cases}
				1 & \text{if}~n=i\\
				0 &\text{otherwise}.
			\end{cases}
		\end{equation*}
	\end{definition}
	
	\begin{remark}\label{delortho}
		We can show that $\{\widehat{\d_{1}},...,\widehat{\d_{N}}\}$ is an orthonormal basis of $\mathbb{C}^{N}.$ In fact, for $i=1,...,N$, 
		\begin{equation}\label{remark}
			\widehat{\d_{i}}(\lambda_{l})=\langle \d_{i}, \c_{l} \rangle=\sum\limits_{n=1}^{N}\d_{i}(n)\overline{\c_{l}(n)}=\overline{\c_{l}(i)}.
		\end{equation}
	\end{remark}
	
	\begin{definition}
		Let $g\in \R$. Then the generalized modulation $M_{i}f$ of $f$ is defined by $$M_{i}f(n)=\chi_{i}(n)f(n),\quad n=1,...,N.$$
	\end{definition}
In classical Fourier analysis, modulations turns out to be the Fourier transform of translations. However, we can not have such a relation in the graph setting.\\

	For further study of graph Fourier transform, generalized translations and generalized modulations, we refer to \cite{MR3440174,ghandehari}.
	\begin{definition}\cite{emerging}
	Let $\widehat{g} : \mathbb{C}\rightarrow \mathbb{C}$ and $\widehat{g_{0}}=\widehat{g}|_{\text{spec}(A)}$. Then one can obtain $g_{0}$ by using the inversion formula for the graph Fourier transform, which we denote by $g$ itself. For such a signal $g$, the dilation is defined as  $$\widehat{D_{s}g}(\lambda_{l})=\widehat{g}(s\lambda_{l}),\quad s\in \mathbb{C},~l=1,...,N.$$
	\end{definition}
\begin{remark}
Notice that unlike the generalized translation and the generalized modulation, the generalized dilation requires the function $\widehat{g}$ to be defined on $\mathbb{C}$ instead of the spectrum of $A.$ We denote the restriction of $\widehat{g}$ to the spectrum of A by $g_{0}$.
\end{remark}
\begin{definition}\cite{hammond}
	Let $\widehat{g} : \mathbb{C}\rightarrow \mathbb{C}$ and $J$ be any finite subset of $\mathbb{C}$. Then the spectral graph wavelet system is defined to be the collection $\{T_{i}D_{s}g : 1\leq i\leq N,~~ s\in J\}.$ 
\end{definition}
	\section{Shift invariant spaces on graphs}
	\begin{definition}
		A subspace $V$ of $\R$ is said to be a shift invariant space if whenever $f\in V$, then $T_{i}f\in V$, for all $i\in \I$.
	\end{definition}
	For $g\in \R$, let $V(g)$ denote the $\Span\{T_{i}g : 1\leq i\leq N\}$. Then $V(g)$ turns out to be a shift invariant space. In fact, given $f\in V(g)$ can be written as $	f=\sum\limits_{i=1}^{N}\alpha_{i}T_{i}g$. Now, for each $1\leq k\leq N,$ using the definition of generalized translation, Kronecker delta function and property $(2)$ in Proposition \ref{convpro}, we get
	\begin{align*}
		T_{k}f&=\sum_{i=1}^{N}\alpha_{i}T_{k}T_{i}g\\
		&=\sqrt{N}\sum_{i=1}^{N}\alpha_{i}(T_{i}g\ast \delta_{k})\\
		&=N\sum_{i=1}^{N}\alpha_{i}(g\ast \delta_{i}\ast \delta_{k})\\
		&=N\sum_{i=1}^{N}\alpha_{i}(g\ast \sum_{m=1}^{N}\langle \delta_{i}\ast \delta_{k}, \delta_{m} \rangle \delta_{m})\\
		&=N \sum_{m=1}^{N}\big(\sum_{i=1}^{N}\alpha_{i}\langle \delta_{i}\ast \delta_{k}, \delta_{m} \rangle \big) g\ast\delta_{m}\\
		&=N \sum_{m=1}^{N}\big(\sum_{i=1}^{N}\alpha_{i}\langle \delta_{i}\ast \delta_{k}, \delta_{m} \rangle \big) T_{m}g,
	\end{align*}	
	which belongs to $V(g)$, proving our assertion.
	\begin{theorem}\label{ort}
		Let $g\in \R$. Then, the collection $\{T_{i}g : 1\leq i\leq N\}$ is an orthonormal basis for $V(g)$ if and only if 
		\begin{equation}\label{orthocondi1}
			|\widehat{g}(\lambda_{l})|=\frac{1}{\sqrt{N}},\quad \text{for all}\quad l\in \{1,...,N\}~(\text{In this case}~V(g)=\mathbb{C}^{N}).
		\end{equation}
	\end{theorem}
	\begin{proof}
		It is enough to show that the system is an orthonormal system, as the dimension of $\mathbb{C}^{N}$ is $N$. Since $T_{i}g=\sqrt{N}(g\ast \delta_{i})$, for $1\leq i\leq N,$ we can write 
		\begin{align*}
			\langle T_{i}g, T_{j}g\rangle&=N\langle g\ast \delta_{i}, g\ast \delta_{j}\rangle\\
			&=N\langle \widehat{g\ast \delta_{i}}, \widehat{g\ast \delta_{j}}\rangle,
		\end{align*}
		by using Parseval identity for the graph Fourier transform. Thus
		\begin{align}\label{orthoij}
			\langle T_{i}g, T_{j}g\rangle 
			&=N\langle \widehat{g}\widehat{\delta_{i}}, \widehat{g}\widehat{\delta_{j}} \rangle\nonumber\\
			&=N\sum\limits_{l=1}^{N}\widehat{g}(\lambda_{l})\widehat{\delta_{i}}(\lambda_{l})\overline{\widehat{g}(\lambda_{l})}\overline{\widehat{\delta_{j}}(\lambda_{l})}\nonumber\\
			&=N\sum\limits_{l=1}^{N}|\widehat{g}(\lambda_{l})|^{2}\widehat{\delta_{i}}(\lambda_{l})\overline{\widehat{\delta_{j}}(\lambda_{l})}.
		\end{align}
		Assume that \eqref{orthocondi1} holds. Then, by using \eqref{orthoij} and Remark \ref{delortho}, it follows that $\{T_{i}g : 1\leq i\leq N\}$ is an orthonormal system. Conversely, assume that $\{T_{i}g \}_{i=1}^{N}$ is an orthonormal system. Then, from \eqref{orthoij}, we have\begin{equation}\label{orthonec}
			N\sum\limits_{l=1}^{N}|\widehat{g}(\lambda_{l})|^{2}\widehat{\delta_{i}}(\lambda_{l})\overline{\widehat{\delta_{j}}(\lambda_{l})}=\begin{cases}
				1 & \text{if}~i= j\\
				0 & \text{otherwise}.
			\end{cases}
		\end{equation} 
		By making use of \eqref{conmatri}, the $(i,j)^{th}$ entry of $\widehat{g}(A)^{\ast}\widehat{g}(A)$ and $\widehat{g}(A)\widehat{g}(A)^{\ast}$ can be computed. It is given by \begin{equation}\label{gmatrix}
			\bigg[\widehat{g}(A)^{\ast} \widehat{g}(A)\bigg]_{i,j}=\bigg[\widehat{g}(A)\widehat{g}(A)^{\ast}\bigg]_{i,j}=\sum\limits_{l=1}^{N}|\widehat{g}(\lambda_{l})|^{2}\chi_{l}(i)\overline{\chi_{l}(j)}.
		\end{equation} By substituting \eqref{remark} and \eqref{orthonec} in \eqref{gmatrix}, we obtain $N\widehat{g}(A)^{\ast} \widehat{g}(A)=N\widehat{g}(A)\widehat{g}(A)^{\ast}=I_{N\times N}$, where $I_{N\times N}$ is the $N\times N$ identity matrix from which \eqref{orthocondi1} follows.
	\end{proof} 
	As a byproduct, we have proved the following result.
	\begin{corollary}
		Let $G$ be an undirected, connected graph with a vertex set $\mathcal{V}$. Let $g : \mathcal{V}\rightarrow \mathbb{C}$ be a signal on $\mathcal{V}$. Then $\{T_{i}g : 1\leq i\leq N\}$ is an orthonormal basis iff the matrix $\sqrt{N}\widehat{g}(A)$ is unitary, where $A$ is the matrix associated with $G$.
	\end{corollary}
	\begin{theorem}
		Let $g,h\in \R$. Then, the two collections $\{T_{i}g : 1\leq i\leq N\}$ and $\{T_{i}h : 1\leq i\leq N\}$ are biorthogonal if and only if \begin{equation}\label{biorthocondi}
			N\overline{\widehat{g}(\lambda_{l})}\widehat{h}(\lambda_{l})=1,\quad \text{for all}\quad l\in \{1,...,N\}.
		\end{equation}
	\end{theorem}
	\begin{proof}
		By using similar steps as in \eqref{orthoij}, we obtain
		\begin{equation}\label{biorthoij}
			\langle T_{i}g, T_{j}h\rangle=N\sum\limits_{l=1}^{N}\widehat{g}(\lambda_{l})\overline{\widehat{h}(\lambda_{l})}\widehat{\delta_{i}}(\lambda_{l})\overline{\widehat{\delta_{j}}(\lambda_{l})}.
		\end{equation}
		Assume that \eqref{biorthocondi} holds. Then, the biorthogonality of the systems $\{T_{i}g : 1\leq i\leq N\}$ and $\{T_{i}h : 1\leq i\leq N\}$ follows from \eqref{biorthoij} and Remark \ref{delortho}. Conversely, assume that $\{T_{i}g : 1\leq i\leq N\}$ and  $\{T_{i}h : 1\leq i\leq N\}$ are biorthogonal. Then, from \eqref{biorthoij}, we have\begin{equation*}
			N\sum\limits_{l=1}^{N}\widehat{g}(\lambda_{l})\overline{\widehat{h}(\lambda_{l})}\widehat{\delta_{i}}(\lambda_{l})\overline{\widehat{\delta_{j}}(\lambda_{l})}=\begin{cases}
				1 & \text{if}~i\neq j\\
				0 & \text{otherwise}.
			\end{cases}
		\end{equation*} 
		In other words, \begin{equation*}
			\bigg[\widehat{g}(A)^{\ast} \widehat{h}(A)\bigg]_{i,j}=\sum\limits_{l=1}^{N}\overline{\widehat{g}(\lambda_{l})}\widehat{h}(\lambda_{l})\chi_{l}(i)\overline{\chi_{l}(j)}.
		\end{equation*}
		This in turn leads to $N\widehat{g}(A)^{\ast} \widehat{h}(A)=I_{N\times N}$, where $I_{N\times N}$ is the $N\times N$ identity matrix, from which \eqref{biorthocondi} follows.
	\end{proof}
\section{Linear independence and orthonormality for a subcollection of a system of generalized translates}
In this section, we obtain necessary and sufficient conditions for a system of generalized translates to form a linearly independent set or an orthonormal set.
\begin{theorem}
	Let $g\in \R$. If $|\widehat{g}(\lambda_{l})|>0$ for all $l\in \I$, then the collection $\V$ is linearly independent, for every $m\leq N$.
\end{theorem}
\begin{proof}
	Since the graph Fourier transform is unitary, it is enough to show that $\{\widehat{g}(\cdot)\widehat{\delta_{i}}(\cdot) : 1\leq i\leq m\}$, is linearly independent in $\R$. Let $l\in \I$ and $m\leq N$. Assume that $\sum_{i=1}^{m}\alpha_{i}\widehat{g}(\lambda_{l})\widehat{\delta_{i}}(\lambda_{l})=0$. This means $\widehat{g}(\lambda_{l})\sum_{i=1}^{m}\alpha_{i}\widehat{\delta_{i}}(\lambda_{l})=0.$
	As $|\widehat{g}(\lambda_{l})|>0$ and $\{\widehat{\delta}_{i} : 1\leq i \leq N\}$ is an orthonormal basis for $\R$, we get $\alpha_{i}=0$ for all $i\in \{1,...,m\}$, thus proving our assertion.
\end{proof}

\begin{theorem}\label{linenecc}
	Let $g\in \R$. Assume that the collection $\V$, $m\leq N$, is linearly independent in $\R$. Then there exists a subcollection $\{l_{1},...,l_{m}\}$ in $\{1,...,N\}$ such that $\widehat{g}(\lambda_{l_{k}})\neq 0$ for all $k=1,...,m$.
\end{theorem}
\begin{proof}
	Since the graph Fourier transform is unitary, $
	\{\widehat{g}(\cdot)\widehat{\delta_{i}}(\cdot) : 1\leq i\leq m\}$ is linearly independent. Consider the $m\times N$ matrix, say $A$, \[\begin{pmatrix}
		\widehat{g}(\lambda_{1})\widehat{\delta_{1}}(\lambda_{1}) & \widehat{g}(\lambda_{2})\widehat{\delta_{1}}(\lambda_{2})& \dots & \widehat{g}(\lambda_{N})\widehat{\delta_{1}}(\lambda_{N})\\
		\widehat{g}(\lambda_{1})\widehat{\delta_{2}}(\lambda_{1}) & \widehat{g}(\lambda_{2})\widehat{\delta_{2}}(\lambda_{2})& \dots & \widehat{g}(\lambda_{N})\widehat{\delta_{1}}(\lambda_{N})\\
		\vdots & \vdots & \ddots & \vdots\\
		\widehat{g}(\lambda_{1})\widehat{\delta_{m}}(\lambda_{1}) & \widehat{g}(\lambda_{2})\widehat{\delta_{m}}(\lambda_{2})& \dots & \widehat{g}(\lambda_{N})\widehat{\delta_{m}}(\lambda_{N})
	\end{pmatrix}
	\] Notice that rank $A=m$. Therefore there exists $m\times m$ submatrix, say $A_{m}$, such that $det(A_{m})\neq 0$. More precisely, there exists a collection $\{l_{1},...,l_{m}\}$ such that \[\begin{vmatrix}
		\widehat{g}(\lambda_{l_1})\widehat{\delta_{1}}(\lambda_{l_1}) & \widehat{g}(\lambda_{l_1})\widehat{\delta_{2}}(\lambda_{l_1})& \dots & \widehat{g}(\lambda_{l_1})\widehat{\delta_{m}}(\lambda_{l_1})\\
		\widehat{g}(\lambda_{l_2})\widehat{\delta_{1}}(\lambda_{l_2}) & \widehat{g}(\lambda_{l_2})\widehat{\delta_{2}}(\lambda_{l_2})& \dots & \widehat{g}(\lambda_{l_2})\widehat{\delta_{m}}(\lambda_{l_2})\\
		\vdots & \vdots & \ddots & \vdots\\
		\widehat{g}(\lambda_{l_m})\widehat{\delta_{1}}(\lambda_{l_m}) & \widehat{g}(\lambda_{l_m})\widehat{\delta_{2}}(\lambda_{l_m})& \dots & \widehat{g}(\lambda_{l_m})\widehat{\delta_{m}}(\lambda_{l_m})
	\end{vmatrix}\neq 0,\] which implies that
	\[\widehat{g}(\lambda_{l_1})\widehat{g}(\lambda_{l_2})\cdots\widehat{g}(\lambda_{l_m})\begin{vmatrix}
		\widehat{\delta_{1}}(\lambda_{l_1}) & \widehat{\delta_{2}}(\lambda_{l_1})& \dots & \widehat{\delta_{m}}(\lambda_{l_1})\\
		\widehat{\delta_{1}}(\lambda_{l_2}) & \widehat{\delta_{2}}(\lambda_{l_2})& \dots & \widehat{\delta_{m}}(\lambda_{l_2})\\
		\vdots & \vdots & \ddots & \vdots\\
		\widehat{\delta_{1}}(\lambda_{l_m}) & \widehat{\delta_{2}}(\lambda_{l_m})& \dots & \widehat{\delta_{m}}(\lambda_{l_m})
	\end{vmatrix}\neq 0,\] and hence $\widehat{g}(\lambda_{l_k})\neq 0$ for all $k\in \{1,...,m\}.$  
\end{proof}

\begin{remark}\label{orthosuff}
	Let $g\in \R$ be such that $|\widehat{g}(\lambda_{l})|=\frac{1}{\sqrt{N}}$ for all $l=1,...,N$. Then the collection $\V$ is an orthonormal system in $\R$, where  $m\leq N$. The proof is exactly similar to the proof of Theorem \ref{ort}. However for the converse we have the following result.
\end{remark}

\begin{theorem}
	Let $g\in \R$. Suppose that the collection $\V$ is an orthonormal system in $\R$. Then there exists a subcollection $l_{1},...,l_{m}$ in $\{1,...,N\}$ such that $|\widehat{g}(\lambda_{l_k})|>0$, for $k=1,...,m.$ In addition, if the subcollection is unique, then  $|\widehat{g}(\lambda_{l_{k}})|=\frac{1}{\sqrt{N}\sqrt{\sum_{q=1}^{m}|\widehat{\delta_{q}}(\lambda_{l_k})|^{2}}}$, where $k=1,...,m$.
\end{theorem}
\begin{proof}
	Since the set $\V$ is an orthonormal system, it is linearly independent in $\R$. Then, by Theorem \ref{linenecc}, there exists a subcollection $l_{1},...,l_{m}$ in $\{1,...,N\}$ such that $|\widehat{g}(\lambda_{l_k})|>0$, for $k=1,...,m.$  By using similar steps as in \eqref{orthoij}, we have $$N\sum_{l=1}^{N}|\widehat{g}(\lambda_{l})|^{2}\widehat{\delta_{i}}(\lambda_{l})\overline{\widehat{\delta_{j}}(\lambda_{l})}=\langle T_{i}g, T_{j}g \rangle_{\R}=\begin{cases}
		1 & \text{if}\hspace{3mm} i=j\\
		0 & \hspace{5mm}\text{otherwise}.
	\end{cases}$$
	In addition, if the subcollection $\{l_1,l_2,...,l_m\}$ is unique, then we have  $$N\sum_{k=1}^{m}|\widehat{g}(\lambda_{l_{k}})|^{2}\widehat{\delta_{i}}(\lambda_{l_{k}})\overline{\widehat{\delta_{j}}(\lambda_{l_{k}})}=\langle T_{i}g, T_{j}g \rangle_{\R}=\begin{cases}
		1 & \text{if}\hspace{3mm} i=j\\
		0 & \hspace{5mm}\text{otherwise}.
	\end{cases}$$ Thus, we get a matrix \[\begin{pmatrix}
		\sqrt{N}\widehat{g}(\lambda_{l_1})\widehat{\delta_{1}}(\lambda_{l_{1}})& \sqrt{N}\widehat{g}(\lambda_{l_1})\widehat{\delta_{2}}(\lambda_{l_{1}}) & \dots \sqrt{N}\widehat{g}(\lambda_{l_1})\widehat{\delta_{m}}(\lambda_{l_{1}})\\
		\sqrt{N}\widehat{g}(\lambda_{l_2})\widehat{\delta_{1}}(\lambda_{l_{2}})& \sqrt{N}\widehat{g}(\lambda_{l_2})\widehat{\delta_{2}}(\lambda_{l_{2}}) & \dots \sqrt{N}\widehat{g}(\lambda_{l_2})\widehat{\delta_{m}}(\lambda_{l_{2}})\\
		\vdots & \vdots & \vdots \\
		\sqrt{N}\widehat{g}(\lambda_{l_m})\widehat{\delta_{1}}(\lambda_{l_{m}})& \sqrt{N}\widehat{g}(\lambda_{l_m})\widehat{\delta_{2}}(\lambda_{l_{m}}) & \dots \sqrt{N}\widehat{g}(\lambda_{l_m})\widehat{\delta_{m}}(\lambda_{l_{m}})
	\end{pmatrix}\] whose column vectors form an orthonormal system in $\mathbb{C}^{m}$. Since it is an $m\times m$ matrix, the corresponding row vectors also form an orthonormal system, which inturn leads to the fact that $N|\widehat{g}(\lambda_{l_{k}})|^{2}\sum_{q=1}^{m}|\widehat{\delta_{q}}(\lambda_{l_{k}})|^{2}=1$ for all $k=1,...,m$, proving our assertion.
\end{proof}
	\section{Shift invariant spaces with multiple generators}
\begin{definition}
	Let $\{g_{s} : 1\leq s\leq M\}$ be a finite set in $\R$. Then the shift invariant space generated by $g_{1},...,g_{M}$, denoted by $V(g_{1},...,g_{M})$, is defined to be $\Span\{T_{i}g_{s} : 1\leq i\leq N,~1\leq s\leq M\}$.\\
	
	 In fact, let $f\in V(g_{1},...,g_{M})$. Then we can write $f=\sum\limits_{i=1}^{N}\sum\limits_{s=1}^{M}\alpha_{i,s}T_{i}g_{s}$. Then proceeding as in Section $3$, we can show that $V(g_{1},...,g_{M})$ is a shift invariant space.
\end{definition}
\begin{theorem}\label{multi}
	Let $\{g_{s} : 1\leq s\leq M\}$ be a finite set in $\R$. Then, the family $\{T_{i}g_{s} : 1\leq i \leq N, 1\leq s\leq M\}$ is a frame for $\R$ with frame bounds $A$ and $B$ if and only if
	\begin{equation}\label{framens}
		\frac{A}{N}\leq \sum\limits_{s=1}^{M}|\widehat{g_{s}}(\lambda_{l})|^{2}\leq \frac{B}{N}\quad\text{for all}~l\in \I.
	\end{equation}
\end{theorem}
\begin{proof}
	By using the Plancherel theorem for the graph Fourier transform and Remark \ref{delortho}, we have
	\begin{align}\label{frameequ}
		\sum\limits_{i=1}^{N}\sum\limits_{s=1}^{M}|\langle f, T_{i}g_{s} \rangle |^{2}&=\sum\limits_{i=1}^{N}\sum\limits_{s=1}^{M}|\langle \widehat{f}, \widehat{T_{i}g_{s}} \rangle |^{2}\nonumber\\
		&=N\sum\limits_{i=1}^{N}\sum\limits_{s=1}^{M}|\langle \widehat{f}, \widehat{g_{s}}\widehat{\d_{i}} \rangle |^{2}\nonumber\\
		&=N\sum\limits_{s=1}^{M}\sum\limits_{i=1}^{N}|\langle \widehat{f}~\overline{\widehat{g_{s}}},\widehat{\d_{i}} \rangle |^{2}\nonumber\\
		&=N\sum\limits_{s=1}^{M}\|\widehat{f}~\overline{\widehat{g_{s}}}\|^{2}\nonumber\\
		&=N\sum\limits_{s=1}^{M}\sum\limits_{l=1}^{N}|\widehat{f}(\lambda_{l})|^{2}|\widehat{g_{s}}(\lambda_{l})|^{2}
	\end{align}
	Assume that $\{T_{i}g_{s} : 1\leq i\leq N, 1\leq s\leq M\}$ is a frame for $\R$, that is, for some $A,B>0$, \begin{equation}\label{framedef}
		A\|f\|^{2}\leq \sum\limits_{i=1}^{N}\sum\limits_{s=1}^{M}|\langle f, T_{i}g_{s} \rangle |^{2} \leq B\|f\|^{2},\quad\text{for all}~f\in \R.
	\end{equation}
	By substituting \eqref{frameequ} in \eqref{framedef}, we get
	\begin{equation}\label{framedef1}
		A\|f\|^{2}\leq N\sum\limits_{l=1}^{N}|\widehat{f}(\lambda_{l})|^{2}\bigg(\sum\limits_{s=1}^{M}|\widehat{g_{s}}(\lambda_{l})|^{2}\bigg)\leq B\|f\|^{2},\quad\text{for all}~f\in \R.
	\end{equation}
	In particular, for each $m=1,...,N$, by choosing $f_{m}\in \R$ such that $\widehat{f_{m}}(\lambda_{m})=1$ and $\widehat{f_{m}}(\lambda_{l})=0$ for $l\neq m$ in \eqref{framedef1}, we obtain our required result. Conversely, assume that \eqref{framens} holds. Then, from \eqref{frameequ}, we have
	$$A\|f\|^{2}\leq \sum\limits_{i=1}^{N}\sum\limits_{s=1}^{M}|\langle f, T_{i}g_{s} \rangle |^{2}\leq B\|f\|^{2}\quad\text{for all}~f\in \R.$$
\end{proof}
As a consequence, we obtain a necessary and sufficient condition for a graph wavelet system to form a frame for $\mathbb{C}^{N}.$
\begin{corollary}\label{dilation}
	Let $g\in \R$ and $J$ be a finite susbset of $\mathbb{R}^{+}$. Then the spectral graph wavelet system $\{T_{i}D_{s}g : 1\leq i\leq N,~s\in J\}$ is a frame for $\R$ with frame bounds $A$ and $B$ if and only if $\frac{A}{N} \leq \sum\limits_{s\in J}|\widehat{g}(s\lambda_{l})|^{2}\leq \frac{B}{N}\quad \text{for all}~l\in \I.$
\end{corollary}
\begin{proof}
	In Theorem \ref{multi}, take $g_{s}=D_{s}g$.
\end{proof}
\begin{example}
	Let $G$ denote the graph given below.
	\begin{figure}[!htb]
	\begin{tikzpicture} 
		\node[draw, circle] (1) at (0,0) {1}; 
		\node[draw, circle] (2) at (2,2) {2}; \node[draw, circle] (3) at (2,-2) {3}; \node[draw, circle] (4) at (-2,0) {4}; 
		 \draw (1) -- (2); \draw (1) -- (3); \draw (1) -- (4); 
		\end{tikzpicture}
	\captionsetup{labelformat=empty}
\caption{Figure : G}
\end{figure}
	The graph Laplacian $\mathcal{L}$ for the graph $G$ can be computed as \begin{equation}
		\begin{pmatrix}
			3 & -1 & -1 & -1\\
			-1 & 1 & 0 & 0\\
			-1 & 0 & 1 & 0\\
			-1 & 0 & 0 &1
		\end{pmatrix}
	\end{equation} By using the matrix equation $\mathcal{L}v=\lambda v$, one can easily show that the eigenvalues and eigenvectors for $\mathcal{L}$ are given by $$\lambda_{1}=0,~\lambda_{2}=\lambda_{3}=1,~\lambda_{4}=4$$ and $\chi_{1}=(1,1,1,1)^{T}$, $\chi_{2}=(0,1,-1,0)^{T}$, $\chi_{3}=(0,1,0,-1)^{T}$, $\chi_{4}=(3,-1,-1,-1)^{T}$ respectively. \\
Consider the Lagrange interpolation polynomial $p$ with respect to the points $(0,1)$, $(1,1)$ and $(4,0)$, given by $p(x)=\frac{1}{12}(-x^{2}+x+12)$.  Let $g\in \R$ be a signal such that $\widehat{g}=p$. Then $\widehat{g}(\lambda_{l})=1$, for $l=1,2,3$ and $\widehat{g}(\lambda_{4})=0$. Let $J$ be any finite subset of $\mathbb{R}^{+}\setminus \mathbb{N}$. Then, it can be easily verified that $\sum\limits_{s\in J}|\widehat{g}(s\lambda_{l})|>0$, for $l=1,2,3,4$. Therefore, by Corollary \ref{dilation}, we can conclude that $\{T_{i}D_{s}g : 1\leq i\leq 4,s\in J\}$ is a frame for $\R$.  
\end{example}
\begin{corollary}
	Let $g\in \R$. Then, the family $\{T_{i}M_{s}g : 1\leq i,s\leq N\}$ is a frame for $\R$ with frame bounds $A$ and $B$ if and only if $\frac{A}{N} \leq \sum\limits_{n=1}^{N}|\chi_{l}(n)|^{2}|g(n)|^{2}\leq \frac{B}{N}\quad \text{for all}~l\in \I.$
\end{corollary}
\begin{proof}
	Now, for $1\leq l\leq N$, by using the graph Fourier transform, we get \begin{align*}
		\sum\limits_{s=1}^{N}|\widehat{M_{s}g}(\lambda_{l})|^{2}
		&=\sum\limits_{s=1}^{N}\bigg|\sum\limits_{n=1}^{N} \chi_{s}(n)g(n)\overline{\chi_{l}(n)}\bigg|^{2}\\
		&=\sum\limits_{s=1}^{N}\bigg(\sum\limits_{n=1}^{N} \chi_{s}(n)g(n)\overline{\chi_{l}(n)}\bigg)\bigg(\sum\limits_{n^{'}=1}^{N} \overline{\chi_{s}(n^{'})}\overline{g(n^{'})}\chi_{l}(n^{'})\bigg)\\
		&=\sum\limits_{n,n^{'}=1}^{N}\chi_{l}(n)\overline{\chi_{l}(n^{'})}g(n)\overline{g(n^{'})}\bigg(\sum\limits_{s=1}\chi_{s}(n)\overline{\chi_{s}(n^{'})}\bigg)\\
		&=\sum\limits_{n=1}^{N}|\chi_{l}(n)|^{2}|g(n)|^{2},
	\end{align*}by using orthonormality of the system $\{\chi_{l} : 1\leq l\leq N\}$.
	Choosing $g_{s}=M_{s}g$ in Theorem \ref{multi}, we obtain our required result. 
\end{proof}
\begin{corollary}
	Let $g\in \R$. Then, the family $\{T_{i}g : 1\leq i\leq N\}$ is a frame for $\R$ with frame bounds $A$ and $B$ if and only if $\sqrt{\frac{A}{N}}\leq |\widehat{g}(\lambda_{l})|\leq \sqrt{\frac{B}{N}}\quad\text{for all}~l\in \I.$ 
\end{corollary}

\begin{theorem}
	Let $\{g_{s} : 1\leq s\leq M\}$ and $\{h_{s} : 1\leq s\leq M\}$ be two finite subsets of $\mathbb{C}^{N}$. Then the collections $\{T_{i}g_{s} : 1\leq i \leq N,~1\leq s\leq M\}$ and $\{T_{i}h_{s} : 1\leq i \leq N,~1\leq s\leq M\}$ are dual frames for $\R$ if and only if
	\begin{equation}\label{dualframecond}
		\sum\limits_{s=1}^{M}\overline{\widehat{g_{s}}(\lambda_{l})}\widehat{h_{s}}(\lambda_{l})=\frac{1}{N}\hspace{3mm}\text{for}\hspace{3mm}1\leq l\leq N.
	\end{equation}
\end{theorem}
\begin{proof}
	By making use of Parseval identity for the graph Fourier transform and property $(1)$ in Proposition \ref{convpro}, we get
	\begin{align*}
	\langle f, T_{i}g_{s} \rangle &= \sqrt{N}\langle \widehat{f}, \widehat{g_{s}}~\widehat{\delta_{i}} \rangle\\
	&=\sqrt{N}\sum_{l=1}^{N}\widehat{f}(\lambda_{l})\overline{\widehat{g_{s}}(\lambda_{l})}\overline{\widehat{\delta_{i}}(\lambda_{l})}.
	\end{align*}
	Similarly we can show that $\langle T_{i}h_{s}, f_{1} \rangle=\sqrt{N}\sum\limits_{l=1}^{N}\overline{\widehat{f_{1}}(\lambda_{l})}\widehat{h_{s}}(\lambda_{l})\widehat{\delta_{i}}(\lambda_{l}).$ Consider
	\begin{align*}\label{dualframeequa}
		\sum\limits_{i=1}^{N}\sum\limits_{s=1}^{M}\langle f, T_{i}g_{s} \rangle \langle T_{i}h_{s}, f_{1}\rangle&=N\sum\limits_{i=1}^{N}\sum\limits_{s=1}^{M} \bigg(\sum_{l=1}^{N}\widehat{f}(\lambda_{l})\overline{\widehat{g_{s}}(\lambda_{l})}\overline{\widehat{\delta_{i}}(\lambda_{l})}\bigg) \bigg(\sum_{l^{'}=1}^{N}\overline{\widehat{f_{1}}(\lambda_{l^{'}})}\widehat{h_{s}}(\lambda_{l^{'}})\widehat{\delta_{i}}(\lambda_{l^{'}})\bigg)\\
		&=N\sum_{l,l^{'}=1}^{N}\sum\limits_{s=1}^{M}\widehat{f}(\lambda_{l})\overline{\widehat{g_{s}}(\lambda_{l})}\overline{\widehat{f_{1}}(\lambda_{l^{'}})}\widehat{h_{s}}(\lambda_{l^{'}})\bigg(\sum_{i=1}^{N}\widehat{\delta_{i}}(\lambda_{l^{'}})\overline{\widehat{\delta_{i}}(\lambda_{l})}\bigg)\\
		&=N\sum_{l,l^{'}=1}^{N}\sum\limits_{s=1}^{M}\widehat{f}(\lambda_{l})\overline{\widehat{g_{s}}(\lambda_{l})}\overline{\widehat{f_{1}}(\lambda_{l^{'}})}\widehat{h_{s}}(\lambda_{l^{'}})\langle \chi_{l},\chi_{l^{'}} \rangle_{\mathbb{C}^{N}}\\
		&=N\sum_{l=1}^{N}\sum\limits_{s=1}^{M} \widehat{f}(\lambda_{l})\overline{\widehat{g_{s}}(\lambda_{l})}\overline{\widehat{f_{1}}(\lambda_{l})}\widehat{h_{s}}(\lambda_{l})\numberthis,
	\end{align*}
using \eqref{remark} and orthonormality of the system $\{\chi_{l} : 1\leq l\leq N\}$.
	Assume that $\{T_{i}h_{s} : 1\leq i \leq N,~1\leq s\leq M\}$ and $\{T_{i}g_{s} : 1\leq i \leq N,~1\leq s\leq M\}$ are dual frames. Then, for every $f,f_{1}\in \R$, we have 
	\begin{equation}\label{dualframe}
		\langle f, f_{1} \rangle =\sum\limits_{i=1}^{N}\sum\limits_{s=1}^{N}\langle f, T_{i}g_{s} \rangle \langle T_{i}h_{s}, f_{1}\rangle.
	\end{equation}
	By substituting \eqref{dualframeequa} in \eqref{dualframe}, we obtain 
	$$\langle \widehat{f}, \widehat{f_{1}} \rangle=N\sum_{l=1}^{N}\sum\limits_{s=1}^{N} \widehat{f}(\lambda_{l})\overline{\widehat{g_{s}}(\lambda_{l})}\overline{\widehat{f_{1}}(\lambda_{l})}\widehat{h_{s}}(\lambda_{l}),\quad\text{for all}~f,f_{1}\in \R.$$ In particular, for every $m=1,...,N$, choose $f=f_{m}$ such that $\widehat{f_{m}}(\lambda_{m})=1$ and $\widehat{f_{m}}(\lambda_{l})=0$ for $l\neq m$. In addition, fix $f_{1}$ such that $\widehat{f_{1}}(\lambda_{l})=1$ for all $l=1,...,N$. Then we get $$1=N\sum\limits_{s=1}^{N}\overline{\widehat{g_{s}}(\lambda_{l})}\widehat{h_{s}}(\lambda_{l}).$$ for all $l\in \{1,...,N\}$.
	Conversely assume that \eqref{dualframecond} holds. Then the result follows from \eqref{dualframeequa}.
\end{proof}
Let $\{u_{i}\}$ be a frame in a separable Hilbert space $\mathcal{H}$. Then a sequence $\{v_{i}\}$ is a dual frame for the sequence $\{u_{i}\}$ if any vector $x$ in $\mathcal{H}$ can reconstructed from its frame coefficients $\{\langle x, u_{i}\rangle\}$. In other words, every vector $x$ can be written as $x=\sum_{i}\langle x, u_i\rangle v_i$. The most natural choice is the canonical dual frame, obtained by applying the inverse of the frame operator to each frame element, but many alternate duals also exist, offering flexibility in applications. We wish to obtain a characterization for a system of generalized translations on $L^{2}(\mathbb{C}^{N})$ to be a dual frame. Towards this end, we obtain the following result. 
\begin{corollary}\label{dual}
	Let $g,h\in \R$. Then the collections $\{T_{i}g : 1\leq i \leq N\}$ and $\{T_{i}h : 1\leq i \leq N\}$ are dual frames for $\R$ if and only if
\begin{equation}
	\overline{\widehat{g}(\lambda_{l})}\widehat{h}(\lambda_{l})=\frac{1}{N}\hspace{3mm}\text{for}\hspace{3mm}1\leq l\leq N.
\end{equation}
\end{corollary}
Let $g\in \mathbb{C}^{N}$ and $\{T_{i}g : 1\leq i\leq N\}$ be a frame for $\mathbb{C}^{N}$. By the general definition, the canonical dual frame associated with $\{T_{i}g : 1\leq i\leq N\}$ is given by $\{S^{-1}T_{i}g : 1\leq i\leq N\}$. In order to study the frame representation, it is necessary to compute the action of $S^{-1}$ on the set $\{T_{i}g : 1\leq i\leq N\}$. The following corollary shows that it is sufficient to determine the image of $g$ under $S^{-1}$, since the remaining elements of the canonical dual frame are the generalized translates of $S^{-1}g$. 
\begin{corollary}
	Let $g\in \mathbb{C}^{N}$ and $\{T_{i}g : 1\leq i\leq N\}$ be a frame for $\mathbb{C}^{N}$. Then the canonical dual frame is given by $\{T_{i}S^{-1}g : 1\leq i\leq N\}$, where $S$ is the frame operator for the system $\{T_{i}g : 1\leq i\leq N\}$. Moreover, the graph Fourier transform of $S^{-1}g$ is $\widehat{(S^{-1}g)}(\lambda_{l})= \frac{1}{N\widehat{g}(\lambda_{l})},\quad\text{for all}~l=1,...,N.$
\end{corollary}
\begin{proof}
	In order to obtain the required structure of the canonical dual frame, we aim to show that $S^{-1}T_{i}=T_{i}S^{-1}$ for $1\leq i\leq N$.
	Let $f\in \mathbb{C}^{N}$. Then, by the definition of frame operator, we get $S(T_{i}f)=\sum\limits_{p=1}^{N}\langle T_{i}f, T_{p}g\rangle T_{p}g$. Now, taking graph Fourier transform on both sides, we get \begin{align*}
		\widehat{S(T_{i}f)}(\lambda_{l})&=\sum\limits_{p=1}^{N}\langle T_{i}f, T_{p}g\rangle \widehat{T_{p}g}(\lambda_{l})\\
	&=\sqrt{N}\sum\limits_{p=1}^{N}\langle \widehat{T_{i}f}, \widehat{T_{p}g}\rangle \widehat{g}(\lambda_{l})\widehat{\delta_{p}}(\lambda_{l})\\
	&=N^{\frac{3}{2}}\sum\limits_{p=1}^{N}\langle \widehat{f}\widehat{\delta_{i}}, \widehat{g}\widehat{\delta_{p}}\rangle \widehat{g}(\lambda_{l})\widehat{\delta_{p}}(\lambda_{l}),
	\end{align*}  by making use of Parseval identity of the graph Fourier transform and property $(1)$ in Proposition \ref{convpro}. Thus
	\begin{align*}\label{st}
		\widehat{S(T_{i}f)}(\lambda_{l})	&=N^{\frac{3}{2}}\sum\limits_{p=1}^{N}\sum\limits_{q=1}^{N}\widehat{f}(\lambda_{q})\widehat{\delta_{i}}(\lambda_{q})\overline{\widehat{g}(\lambda_{q})}\overline{\widehat{\delta_{p}}(\lambda_{q})}\widehat{g}(\lambda_{l})\widehat{\delta_{p}}(\lambda_{l})\\
		&=N^{\frac{3}{2}}\sum\limits_{q=1}^{N}\widehat{f}(\lambda_{q})\widehat{\delta_{i}}(\lambda_{q})\overline{\widehat{g}(\lambda_{q})}\widehat{g}(\lambda_{l})\bigg(\sum\limits_{p=1}^{N}\overline{\widehat{\delta_{p}}(\lambda_{q})}\widehat{\delta_{p}}(\lambda_{l})\bigg)\\
		&=N^{\frac{3}{2}}\sum\limits_{q=1}^{N}\widehat{f}(\lambda_{q})\widehat{\delta_{i}}(\lambda_{q})\overline{\widehat{g}(\lambda_{q})}\widehat{g}(\lambda_{l})\langle \chi_{q}, \chi_{l} \rangle_{\mathbb{C}^{N}}\\
		&=N^{\frac{3}{2}}\widehat{f}(\lambda_{l})\widehat{\delta_{i}}(\lambda_{l})|\widehat{g}(\lambda_{l})|^{2}\numberthis.
	\end{align*}
On the other hand,
\begin{align*}
\widehat{T_{i}(Sf)}(\lambda_{l})&=\sum\limits_{p=1}^{N}\langle f, T_{p}g \rangle \widehat{T_{i}T_{p}g}(\lambda_{l})\\
&=N\sum\limits_{p=1}^{N}\langle \widehat{f}, \widehat{T_{p}g} \rangle \widehat{g}(\lambda_{l})\widehat{\delta_{i}}(\lambda_{l})\widehat{\delta_{p}}(\lambda_{l}).
\end{align*}
Then, proceeding as before, it is easy to show that
\begin{align*}\label{ts}
	\widehat{T_{i}(Sf)}(\lambda_{l})&=\sum\limits_{p=1}^{N}\langle f, T_{p}g \rangle \widehat{T_{i}T_{p}g}(\lambda_{l})\\
	&=N\sum\limits_{p=1}^{N}\langle \widehat{f}, \widehat{T_{p}g} \rangle \widehat{g}(\lambda_{l})\widehat{\delta_{i}}(\lambda_{l})\widehat{\delta_{p}}(\lambda_{l})\\
	&=N^{\frac{3}{2}}\widehat{f}(\lambda_{l})|\widehat{g}(\lambda_{l})|^{2}\widehat{\delta_{i}}(\lambda_{l}).\numberthis
\end{align*} 
Thus from \eqref{st} and \eqref{ts}, we can conclude that $ST_{i}=T_{i}S$. Since $S$ is an invertible, we obtain $T_{i}S^{-1}=S^{-1}T_{i}.$ By applying Corollary \ref{dual} for the system $\{T_{i}g : 1\leq i \leq N\}$ and its canonical dual frame $\{T_{i}S^{-1}g : 1\leq i \leq N\}$, we obtain our required result.
\end{proof}

\section{Acknowlegdement}
 The author (R .V) thanks Indian Institute of Technology Madras, India, for the financial support. The author (R .R) thanks NBHM, DAE, India, for the research project grant.
	
	\vspace{10mm}
	
	
	
	
	
	\vspace{5mm}
	
	\bibliographystyle{amsplain}
	\bibliography{GSP}
\end{document}